\documentclass[11pt,leqno]{amsart}
\usepackage[english]{babel}
\usepackage{color}
\usepackage{hyperref}

\newtheorem{te}{Theorem}[section]

\newtheorem{co}{Corollary}[section]
\newtheorem{lm}{Lemma}[section]
\newtheorem{de}{Definition}[section]

\usepackage{color}

\usepackage{color}

\begin{document}

\title[New Solutions to the $G_2$ Hull-Strominger System]{New Solutions to the $G_2$ Hull-Strominger System via torus fibrations over $K3$ orbifolds}

\author{Anna Fino, Gueo Grantcharov, Jose Medel}
\address{Dipartimento di Matematica \lq\lq Giuseppe Peano\rq\rq \\ Universit\`a di Torino\\
Via Carlo Alberto 10\\
10123 Torino\\ Italy\\
and Department of Mathematics and Statistics Florida International University\\
  Miami Florida, 33199, USA}
 \email{annamaria.fino@unito.it, afino@fiu.edu}
 \address{ Department of Mathematics and Statistics Florida International University\\
  Miami Florida, 33199, USA\\
  and Institute of Mathematics and Informatics, Bulgarian Academy of Sciences,\\
  8 Acd. G. Bonchev str.\\
 1113 Sofia\\ Bulgaria}
\email{grantchg@fiu.edu, grantchg@math.bas.bg}
\address{ Department of Mathematics and Statistics Florida International University\\
	Miami Florida, 33199, USA}
\email{ jmede022@fiu.edu}

\maketitle

\begin{abstract} Using torus fibrations over K3 orbisurfaces, we construct new smooth solutions to the 
$G_2$ Hull–Strominger system. These manifolds arise as total spaces of principal $T^3$
 (orbi)bundles over singular K3 surfaces. Our construction is based on the choice of three divisors on a singular K3 surface that are primitive with respect to a particular K\"ahler metric. The stable bundle is obtained via an adaptation of the Serre construction  to the singular setting.

\end{abstract}
\section{Introduction} 

Riemannian manifolds admitting  metric connections with totally skew-symmetric torsion and special holonomy have been studied in both mathematics and theoretical physics. The Killing spinor equations in heterotic supergravity require the existence of a spinor which is parallel with respect to a connection with torsion $H$ corresponding to an NS-flux.  On  manifolds of dimension 6, 7, or 8, this leads to a  restriction of the holonomy of such connection to $SU(3)$, $G_2$, $Sp(2)$, or $Spin(7)$ respectively. When the torsion does not vanish, it has to satisfy the modified Bianchi identity associated with the Green--Schwarz anomaly cancellation condition in heterotic string theory.  To  yield a supersymmetric vacuum of the heterotic theory,  a solution of the Killing spinor equations satisfying  the Bianchi identity,  must also  satisfy the instanton condition (\cite{Ivanov10}).

The resulting system of equations is  first considered on 6-manifolds and is called the Hull-Strominger system, introduced in \cite{Hull, Strominger}, which  describes the geometry of compactification of heterotic superstrings with torsion to the $4$-dimensional Minkowski spacetime. For a manifold of dimension seven, this system is also known as the $G_2$ Hull-Strominger system (or heterotic $G_2$ system) and has been investigated extensively in \cite{CGT,  DLS, DLS18,  FIUV11, FIUV15, GMW,  GL, II, GS25, GMPS, KPS, MSS, OLMS}. More recently, it has been  generalized in \cite{ADGL}, motivated by  theoretical physics and generalized geometry.

In this paper we focus on the 7-dimensional case. The 7-dimensional compact manifolds relevant to heterotic compactifications must admit  an integrable $G_2$-structure  \cite{GKMW, FI03},  also called  $G_2$-structure
with torsion (shortly $G_2T$-structure),  which guarantees the existence of a unique connection with totally skew symmetric torsion \cite{FI02}.   We recall that  a $G_2$-structure on a seven-dimensional manifold  $M$ is  defined by a globally defined $3$-form $\varphi$  that satisfies a certain non-degeneracy condition. This 3-form induces a Riemannian metric $g_{\varphi}$  and a volume form ${\text{Vol}}_{\varphi}.$ The study of $G_2$-structures was initiated by Fern\'andez and Gray \cite{Fernandez-Gray}, who classified them based on the irreducible components of the covariant derivative of $\varphi$ with respect to the Levi-Civita connection.  A  $7$-manifold $M$ endowed with a $G_2$-structure $\varphi$   admits a  connection  $\nabla^{\varphi}$  preserving the $G_2$-structure  and   with  totally skew symmetric torsion   if and only
if  \begin{equation}\label{integrableG2}  d \star  \varphi = \theta \wedge \star \varphi, \end{equation} where $\star$  is  the Hodge star operator  with respect to  the metric $g_{\varphi}$  induced by $\varphi$ and $\theta$ denotes the  $G_2$ Lee form.   A  $G_2$-structure  is called  integrable  (or $G_2T$)  if the condition \eqref{integrableG2}  is satisfied. The torsion $3$-form $H_{\varphi}$ can be written in terms of $\varphi$ as follows
\begin{equation}\label{eqn:tor}
H_{\varphi} = \frac{1}{6} {\star( d\varphi \wedge  \varphi)}  \varphi - \star d\varphi + \star(\theta \wedge \varphi), 
\end{equation}
where $\theta = -\tfrac13\star(\star d\varphi\wedge \varphi)$.  For recent results in  the case of closed torsion see for instance \cite{FMR, FF}.  A classification of  Riemannian $7$-manifolds carrying  a $G_2 T$-structure  with parallel  torsion  has been obtained  \cite{MS}, up to naturally reductive homogeneous spaces and nearly parallel  $G_2$-structures.

Given a $7$-dimensional compact spin manifold $M$ and a principal $G$-bundle $P \to M,$ 
the $G_2$ Hull-Strominger system for a $G_2$-structure $\varphi$ on $M,$ 
a smooth function $f \in \mathcal{C}^\infty(M)$, a connection $\nabla$ on $TM$ and a connection $A$ on $P$ can be written as follows
\begin{equation}
\begin{split} \label{StromingersystemG2}
d\varphi \wedge \varphi &= 0, \quad  d\star \varphi=\, -4 df \wedge \star \varphi,\\
F_A \wedge \star \varphi & = 0,\quad R_\nabla \wedge \star \varphi= 0,\\
d H_{\varphi} &= \frac{\alpha'}{4}\left(\mathrm{tr}(F_A\wedge F_A) - \mathrm{tr}({R_\nabla \wedge R_\nabla})\right), 
\end{split}
\end{equation}
where $F_A$ and $R_\nabla$ are the curvatures of $A$ and $\nabla$, respectively, $\alpha'$ is a positive constant  
(the square of the string length) and $H_{\varphi}$ is the torsion of the $G_2$-structure  $\varphi$ given by
$$
H_{\varphi} = \star(- d\varphi -4df \wedge \varphi ).
$$
Thus, any definite $3$-form $\varphi$ solving  the system \eqref{StromingersystemG2} is a $G_2T$-structure, it satisfies the additional constraint $d\varphi \wedge \varphi=0$, and 
its Lee form $\theta = -4df$  must be exact. The second line of equations in \eqref{StromingersystemG2} is the $G_2$-instanton condition, and has been
the subject of important recent results (see  for instance  \cite{DS,HN, Oliveira, SW}).

The last equation in \eqref{StromingersystemG2} is called the Bianchi identity
and arises from the Green--Schwarz anomaly cancellation condition.
In heterotic string theory, the curvature term
$R^\nabla\wedge R^\nabla$ is naturally defined using the Hull connection.
In this paper, we instead work with an arbitrary $G_2$-instanton connection
$\nabla$ on $TM$, as in \cite{CGT}.

As shown in \cite{GMPW} (see also \cite[Prop.~3.1]{CGT}),  when $M$ is compact and $dH_{\varphi} =0$, any solution to the system
\[
d\varphi \wedge \varphi=0, \qquad
d\star \varphi= -4 df\wedge \star \varphi,  \qquad 
d H_{\varphi} = 0, \quad 
\]
must have constant $f$ and $H_{\varphi}=0$. Consequently,  in particular the $G_2$-structure must be parallel,  namely  $d \varphi =0$ and $d \star \varphi =0$.  In this case compact solutions to the system \eqref{StromingersystemG2}   are constructed considering   the bundle of orthogonal frames of  the  $G_2$-holonomy metric,
and  $A =\nabla$ also equal to the Levi-Civita connection.
The first compact solutions with non-zero torsion $H_{\varphi}$ (and constant dilaton function $f$) to the  system \eqref{StromingersystemG2}   have been
constructed in \cite{FIUV11}. Non-compact solutions to \eqref{StromingersystemG2}  have been constructed in  \cite{FIUV15, GN}. In  \cite{CGT}, following  the  method initiated by Fu and Yau  \cite{FuYau}  for the construction of the first solutions to the Hull-Strominger system on non-K\"ahler complex $3$-folds, Clarke, Garcia-Fernandez and Tipler provide  new solutions  to the  system \eqref{StromingersystemG2}.  The  solutions  to the Hull-Strominger system on non-K\"ahler complex $3$-folds  require that the connection  that appears in the anomaly cancellation term is the Chern connection on the $3$-fold.  With the diﬀerent hypothesis that  the connection  is an instanton,
Garcia-Fernandez  constructed  new solutions to the Hull-Strominger system on the same
3-folds  \cite{Fernandez18}. In  \cite{CGT}   this method  has been extended 
to the 7-dimensional case (see also \cite[Section 6]{FIUV11}). The moduli space of solutions to the system  \eqref{StromingersystemG2} has been widely studied \cite{DLMMS, DLS16, DLS18, DLS18bis, FQS, CGT}. The study is broadly based on the equivalence of the Hull-Strominger system to exactness of a sequence of differential operators, and the infinitesimal moduli correspond to the first cohomology of this sequence. 

In the present paper we construct examples of solutions of the 7-dimensional Hull-Strominger system with different topology. They are principal $T^3$-(orbi)bundles over singular K3 surfaces. The examples generalize the ones in \cite{CGT, FIUV11, FIUV15}; see also \cite{YZ} for constructions of closed $G_2$-structures with $T^3$-symmetry.  Following the preliminary discussion of K3 orbifolds and their associated Seifert bundles in Section \ref{section2}, we establish the framework for constructing our spaces in Section \ref{section3}.  In particular it requires  three  divisors on a singular K3 surface which are primitive with respect to a particular K\"ahler metric, as well as a stable bundle with some restriction on its characteristic numbers. The construction of such objects on particular singular K3 surfaces from the lists of Iano-Fletcher \cite{IF} and M. Reid \cite{R} is presented  in Section  \ref{section4}. The construction of the stable bundle is based on the Serre construction of vector bundles adapted to the singular case at hand. 
Finally we combine our results to prove the main Theorem in Section 5, and in Section 6 we present  explicit examples, together with some topological information.

 \section{Preliminaries on K3 Orbifolds and $S^1$ Seifert bundles} \label{section2}

We recall that a  complex orbifold $X$  is  a normal complex space 
locally given by charts written as quotients of smooth coordinate charts. More precisely,  $X$  is a topological space together with a cover of coordinate charts where each element of the cover is homeomorphic to a quotient of an open in $\mathbb{C}^n$ containing the origin by a finite group $G_x$ and the transition functions are covered by holomorphic maps. In the paper  we will consider complex orbifolds with cyclic isotropy groups $G_x$, where $x$ is the image of the origin which is fixed under the action. The (complex) dimension of the orbifold  is the number $n$ in $\mathbb{C}^n/G_x$.  When  $n=2$ the orbifold  is  also called {\bf orbisurface}. Examples of  orbisurfaces which arise as  complete intersections in weighted projective spaces are given in \cite{IF}.

\begin{de}\label{definition2} An {\bf $A_n$ singularity} in a complex orbisurface  $X$ is an isolated  singular point $p\in X$ such that there is a neighborhood $W$ of $p$ with a chart $W\cong \mathbb{C}^2/\mathbb{Z}_{n+1}$, or equivalently, $W\cong V(x^2+y^2+z^{n+1}),$ where by $V(x^2+y^2+z^{n+1})$ we denote  the vanishing locus of the polynomial  $x^2+y^2+z^{n+1}$.
\end{de}

In the paper  we will consider  orbifold analogues of K3 surfaces. 

\begin{de}\label{definition1} A {\bf singular K3}  surface $X$ is a surface with only isolated singularities $\Sigma=\{p_1,...,p_k\}$, such that $H^1(X,\mathcal{O}_X)=0$ and $\omega_X\cong \mathcal{O}_X$, where $\omega_X:=i_*\omega_{X_{reg}}$, $i:X_{reg}\rightarrow X$ is the inclusion, and $\omega_{X_{reg}}$ is the canonical sheaf of $X_{reg},$ where $X_{reg}=X\setminus \Sigma$. Moreover, if we endow $X$ with an orbifold structure $\mathcal{X}=(X,\mathcal{U})$ we call $\mathcal{X}$ a {\bf K3 orbisurface}.
\end{de}

Note that the dualizing sheaf of $X$  is isomorphic to $i_*\omega_{X\setminus \Sigma}$, so the definition is consistent. Suitable K3 orbisurfaces with at worst finite isolated $A_n$ singularities can be found in \cite{R}  and \cite{IF}. Their orbifold structure is inherited from their respective weighted projective spaces.  As the smooth K3 surfaces admit a hyper-K\"ahler metric, the K3 orbifolds admit singular hyper-K\"ahler metrics of orbifold type. This follows from the solution of the Calabi-Yau problem for such orbifolds  (see for instance \cite{Faulk} and the references therein). We'll use this fact in the later sections.

In the  next  section  we  will consider  $T^3$-bundles  over an orbisurface  $X$. The  general theory of such spaces from the foliations view-point is given in \cite{HS}. In particular, such bundles are determined by three  rational divisors on the base orbifold and are  constructed via Seifert $S^1$-bundles.  For the definition of  Seifert $S^1$-bundle we refer to \cite[Definition 4.7.6]{BG}. Roughly speaking,  Seifert  $S^1$-bundles  are  spaces  with  a locally free $S^1$-action, for which the $S^1$-foliation has an orbifold leaf space.  A  multiple leaf  is an $S^1$-orbit on which the action is not globally free. 

We recall the following result, which is  Theorem 4.7.3 in \cite{BG} and has been  proven by  Koll\'ar in \cite{Kollar2004}.

\begin{te} \label{theor3.1}
Let $X$ be a normal reduced complex space with at worst quotient singularities and $\Delta=\sum_i (1-\frac{1}{m_i})D_i$
be a $\mathbb{Q}$ divisor (this is the data associated to an orbifold). Then there is a one-to-one correspondence between Seifert ${\mathbb C}^*$-bundles $f:Y\rightarrow (X,\Delta)$ and the  following data:

(i) For each $D_i$ an integer $0\leq b_i< m_i$ relatively prime to $m_i$, and

(ii) a linear equivalence class of Weil divisors $B\in Div(X)$.
\end{te}

In the paper we will   consider the case when $Y$ is smooth and we will need  to consider the \lq\lq smoothness part\rq\rq \, of Theorem 4.7.7 in \cite{BG}:

\begin{te}

If $(X,\Delta)$ is a locally cyclic orbifold as in the Theorem above and $f: Y \rightarrow (X,\Delta)$  is an $S^1$-orbibundle whose local uniformizing groups inject into the group $S^1$ of the orbibundle, then $f: Y \rightarrow (X,\Delta)$ is a Seifert $S^1$-bundle and $Y$  is smooth.

\end{te}

For algebraic orbifolds this could be refined (see  Theorem 4.7.8 in \cite{BG} and \cite{K05}) and for an orbifold $(X,\Delta)$ with trivial $H^1_{orb}(X,{\mathbb Z})$ the Seifert $S^1$-bundle  $Y$  is uniquely determined by its first Chern class $c_1(Y/X)\in H^2(X,{\mathbb Q})$, which is defined as $[B] + \sum_i \frac{b_i}{m_i} [D_i]$.

\section{An ansatz on $T^3$-fibrations over $K3$ orbifolds} \label{section3}

Let  $(X,\omega_1, \omega_2, \omega_3)$  be  a compact K3  orbisurface   equipped  with  the hyperk\"ahler triple  $(\omega_1, \omega_2, \omega_3)$ of 2-forms   each of pointwise length $\sqrt{2}$. Assume that $X$ admits  three  anti-self-dual $(1,1)$-forms $\beta_1, \beta_2$ and $\beta_3$   such that $[\beta_1], [\beta_2], [\beta_3] \in H^2_{orb}(X,\mathbb{Z})$  (up to the factor $\frac{1}{2 \pi}$)   and the total space $M$ of the principal $T^3$  orbifold bundle
$\pi: M \rightarrow  X$ determined by them is smooth.

Let  $\theta = (\theta_1, \theta_2, \theta_3)$ be the $T^3$-connection form on
$M$, with values in $\mathbb R^3$, that satisfies $d \theta = (\pi^* \beta_1,  \pi^* \beta_2,  \pi^* \beta_3)$.  Let  $u$ be the pullback to $M$ of a smooth real-valued  function on $X$,
 and let $t > 0$ be a  constant.  As in \cite{CGT} we can consider the $3$-form
$$
\varphi := \varphi_{u,t} =  t^3 \theta_1 \wedge \theta_2 \wedge \theta_3  - t e^u  ( \theta_1  \wedge \omega_1 + \theta_2  \wedge  \omega_2 + \theta_3  \wedge \omega_3).
$$
The induced metric $g_{\varphi}$  and
volume form $ {\text{Vol}}_{\varphi}$  are  respectively given by
$$
g_{\varphi} = t^2 \sum_{i = 1}^3 \theta_i^2 + e^u \pi^* g_X, \quad  {\text{Vol}}_{\varphi} = t^3 e^{2u}  \theta_1 \wedge \theta_2 \wedge \theta_3  \wedge  \pi^*{\text{Vol}}_X.
$$
By Proposition 4.1 in \cite{CGT} it follows that $d \varphi \wedge \varphi =0$ and by 
Lemma 4.2 in \cite{CGT} we have that the Hodge dual $4$-form  $\star \varphi$  is given by
$$
\star \varphi =  \frac{e^{2u}}{2} \omega_1 \wedge \omega_1 - t^2 e^u ( \theta_2 \wedge \theta_3  \wedge \omega_1  +  \theta_3 \wedge \theta_1 \wedge \omega_2 +  \theta_1 \wedge \theta_2   \wedge \omega_3).
$$
By  Proposition  4.3  in \cite{CGT} we still  have
$$
d( \star \varphi) = - t^2 e^u  du  \wedge  (\theta_2 \wedge \theta_3 \wedge  \omega_1  + \theta_3 \wedge \theta_1 \wedge \omega_2 + \theta_1 \wedge \theta_2 \wedge \omega_3)
$$  and so $d (\star \varphi) =   -4 df\wedge \star \varphi$ with $f = - \frac 14 u$.  By Lemma 4.4 in \cite{CGT} the torsion $3$-form  $H_{\varphi}$ is given by
$$
H_{\varphi} = t^2  (\beta_1  \wedge  \theta_1 + \beta_2  \wedge  \theta_2 + \beta_3 \wedge  \theta_3) - \frac 12 e^u \iota_{{\text{grad}}(u)} \,  \omega_1^2,
$$
where ${{\text{grad}}(u)}$ is the gradient of $u$ on $X$
 and  by Lemma 4.5 in \cite{CGT}  we have
 $$
 dH_{\varphi}=  t^2 ( \beta_1^2  +  \beta_2^2 + \beta_3^2)  -\frac 12 e^u du \wedge  \iota_{{\text{grad}}(u)} \,  \omega_1^2 - \frac 12 e^u d  (\iota_{{\text{grad}}(u)}  \, \omega_1^2).
 $$
In the next section we will consider as instanton data a connection $A$ on a vector bundle instead of a principal bundle. We recall the notion of hyperholomorphic bundle over a hyper-K\"ahler manifold \cite{Verbitsky}. A Hermitian connection $A $ on the Hermitian
vector bundle $E$ is {\bf hyperholomorphic} if the curvature $F_A$  is of type $(1, 1)$ with respect to the
three complex structures  associated to the K\"ahler forms  $\omega_i$. In real dimension 4, this is equivalent to $F_A$ being anti-self-dual, and $A$ is called an instanton.  Note that the definition of hyperholomorphic bundle and Proposition 4.6 in \cite{CGT}  can be extended to K3 orbifolds, since the hyperholomorphic condition for a vector (or principal) bundle does not  depend on the smoothness if the orbifold singularities are isolated and DuVal.

\section{Divisors and stable bundles on K3 orbisurfaces} \label{section4}

	Our construction of solutions to equations (\ref{StromingersystemG2}) is based on the following Lemma. Note that in what follows we use Cartier divisor class and element of the Picard group interchangeably.

\begin{lm}\label{lm4.1}
    Let $X$ be a normal surface with at worst isolated $A_n$-singularities, and assume that at least one of them has $n \geq 3$. Suppose that there is an ample Cartier divisor $\Bar{H}$ such that the Seifert $S^1$-bundle defined by $c_1=\bar{H}$ is smooth and simply connected. Then there is a chain of blow-ups $\pi:\Tilde{X}\rightarrow X$ such that there exist a rational ample divisor $\mathcal{E}\in \text{Pic}(\Tilde{X})\otimes \mathbb{Q}$ and three divisors $\mathcal{D}_0,\mathcal{D}_1,\mathcal{D}_2\in\text{Pic}(\Tilde{X})$ such that
    \begin{enumerate}
        \item For all $i$, $\mathcal{D}_i$ is primitive with respect to $\mathcal{E}$,  i.e. $\mathcal{D}_i\cdot \mathcal{E}=0$.
      \item The Seifert $S^1$-bundles $f_0:Y_1\rightarrow \Tilde{X}$ with $c_1(Y_1/\Tilde{X}) = \mathcal{D}_0$, $f_1:Y_2\rightarrow Y_1$  with $c_1(Y_2/Y_1) = f_0^*\mathcal{D}_1$, and $f_2:Y_3\rightarrow Y_2$ with $c_1(Y_3/Y_2) = (f_0\circ f_1)^*\mathcal{D}_2$ are smooth and simply connected.    
 
    \end{enumerate}

\end{lm}

\begin{proof}
   The chain of blow-ups $\pi: \widetilde{X} \rightarrow X$ can be decomposed in the following way: $\pi=\pi_0\circ \pi_1$, where $\pi_0$ is an arbitrary chain of blow-ups of singular points of $X$ and of the subsequent singular points; in particular, it may be used to reduce a singularity of type $A_n$ with $n>4$ to an $A_3$ or $A_4$ if needed, and $\pi_1$ is the full resolution of the resulting $A_3$ or $A_4$ singularity. Observe that for $\pi_0:\bar{X}\rightarrow X$ we can define an ample divisor for $\bar{X}$ in  the following way: let $C_i$ be the irreducible negative curves which appear in the map $\pi_0,$ so there is an ample divisor 
    $$\bar{\mathcal{E}}:=\bar{H}-\frac{1}{j}\left(\sum_i C_i\right).$$
    It is easy to see that for a $j\in \mathbb{N}$ big enough this divisor is ample by Nakai-Moishezon criteria of ampleness.

    First we check the case when $\pi_1:\Tilde{X}\rightarrow \bar{X}$ is the resolution of an $A_3$ singularity. Thus $\text{Pic}(\Tilde{X})\cong \text{Pic}(\bar{X})\oplus \langle \mathcal{L},\mathcal{L}',\mathcal{G}\rangle$,  where $\mathcal{G}=C_0$, $\mathcal{L}=2C_1,$ and $\mathcal{L}'=C_1+C_2$. $C_1$ and $C_2$ are two irreducible curves in the exceptional divisor of the first blow-up in the resolution of the $A_3$ singularity,  and $C_0$ is the exceptional divisor of the second  blow-up of the resolution. Moreover, $C_0$ is the exceptional curve of the blow-up of an $A_1$ singularity. Observe that $C_0^2=-2$, $C_1^2=C_2^2=-\frac{3}{2}$, $C_1\cdot C_2=\frac{1}{2}$, and $C_0\cdot C_1=C_0\cdot C_2=0$. Let $c=\bar{H}^2$, and $\bar{\mathcal{E}}$ be an ample divisor over $\bar{X}$ defined as above, so $\bar{\mathcal{E}}\cdot \pi_0^*\bar{H} = \bar{H}^2=c$.  We define 
    \begin{align*}
    \mathcal{E}&:=\frac{k(3a-b)(3b-a)}{c}\pi_1^*\bar{\mathcal{E}}-aC_1-bC_2-\frac{1}{2}C_0,\\
    \mathcal{D}_0&:=\pi^*\bar{H}-(k(3b-a)-1)(3a-b)C_0-2C_1,\\
    \mathcal{D}_1&:=\pi^*\bar{H}-2k(3b-a)C_1,\\
    \mathcal{D}_2&:=\pi^*\bar{H}-2k(3a-b)C_2,
    \end{align*}
    where $a=3m+c-1$, $b=c+m$,  with $k,m\in \mathbb{N}$  to be determined later. We assumed that $\bar{H}$ is Cartier. $C_0$ is Cartier, and $C_1, C_2$ are $\mathbb{Q}$-Cartier, and $2C_1$ and $2C_2$ are Cartier. This means that $\mathcal{D}_0, \mathcal{D}_1, \mathcal{D}_2$ are Cartier since they are integer linear combinations of Cartier divisors.  Observe that $3a>b$ and $3b>a$.\\
    \vspace{.05in}
    We now check the ampleness of $\mathcal E$ using the
Nakai--Moishezon criterion. We have to prove that
$\mathcal E\cdot\Gamma>0$ for every irreducible curve
$\Gamma\subset\widetilde X$ and that $\mathcal E^2>0$.

Let $\widetilde D\subset\widetilde X$ be the strict transform of an
irreducible curve $D\subset\bar X$. If $D$ does not pass through the
$A_3$ singularity, then
$\widetilde D=\pi_1^*D,$
and hence
\[
\mathcal E\cdot\widetilde D
=q\,\bar{\mathcal E}\cdot D>0,
\qquad
q=\frac{k(3a-b)(3b-a)}{c}.
\]

Suppose now that $D$ passes through the $A_3$ singularity. Its strict
transform can be written as
\[
\widetilde D
=\pi_1^*D-\mu_0C_0-\mu_1C_1-\mu_2C_2,
\]
where the $\mu_i$ are the multiplicities along the successive exceptional
divisors. Therefore,
\[
\mathcal E\cdot\widetilde D
=q\,\bar{\mathcal E}\cdot D
-\sum_{i=0}^2\mu_i\,\mathcal E\cdot C_i.
\]
Since the resolution is fixed and $\bar{\mathcal E}$ is ample, there is
a constant $A>0$, depending only on the chosen resolution,  such that
\[
\sum_{i=0}^2\mu_i\,\mathcal E\cdot C_i
\leq A\,\bar{\mathcal E}\cdot D.
\]
Indeed, after multiplying $\bar{\mathcal E}$ by a sufficiently divisible
positive integer, one may assume that it is very ample, and the
multiplicities at the singular point and at the successive infinitely
near points are bounded by a fixed multiple of
$\bar{\mathcal E}\cdot D$. Consequently,
\[
\mathcal E\cdot\widetilde D
\geq(q-A)\bar{\mathcal E}\cdot D>0
\]
for $k$, and hence $q$, sufficiently large.

It remains to check the curves in the exceptional divisor. We have
\begin{align*}
\mathcal E\cdot C_0
 &=1>0,\\
\mathcal E\cdot C_1
 &=-aC_1^2-bC_2\cdot C_1
   =\frac{3a-b}{2}>0,\\
\mathcal E\cdot C_2
 &=-aC_1\cdot C_2-bC_2^2
   =\frac{3b-a}{2}>0.
\end{align*}

Finally,
\begin{align*}
\mathcal E^2
 &=q^2\bar{\mathcal E}^{\,2}
   +a^2C_1^2+b^2C_2^2
   +2ab\,C_1\cdot C_2
   +\frac14C_0^2\\
 &=q^2\bar{\mathcal E}^{\,2}
   -\frac32a^2-\frac32b^2+ab-\frac12>0
\end{align*}
for $k$, and hence $q$, sufficiently large.  Since we have shown that
$\mathcal E\cdot\Gamma>0$ for every irreducible curve
$\Gamma\subset\widetilde X$ and that $\mathcal E^2>0$, the
Nakai--Moishezon criterion implies that $\mathcal E$ is ample.
    \vspace{.05in}
    Now we show that for all $i$, $\mathcal{D}_i$ is primitive with respect to  $\mathcal{E}$. 
Since
$\bar{\mathcal E}\cdot\pi_0^*\bar H=\bar H^2=c,$
we obtain
    \begin{align*}
        \mathcal{E}\cdot \mathcal{D}_0&=k(3a-b)(3b-a)+\frac{1}{2}(k(3b-a)-1)(3a-b)C_0^2+2aC_1^2+2bC_1\cdot C_2\\
        &=k(3a-b)(3b-a)+\frac{1}{2}(k(3b-a)-1)(3a-b)(-2)+2a\left(-\frac{3}{2}\right)+2b \left(\frac{1}{2}\right)\\
        &=k(3a-b)(3b-a)-(k(3b-a)-1)(3a-b)-(3a-b)\\
        &=0,\\
        \mathcal{E}\cdot \mathcal{D}_1&=k(3a-b)(3b-a)+2ka(3b-a)C_1^2+2kb(3b-a)C_1\cdot C_2\\
        &=k(3a-b)(3b-a)+2ka(3b-a)\left(-\frac{3}{2}\right)+2kb(3b-a) \left(\frac{1}{2}\right)\\
        &=0,\\
        \mathcal{E}\cdot \mathcal{D}_2&=k(3a-b)(3b-a)+2kb(3a-b)C_2^2+2ka(3a-b)C_1\cdot C_2\\
        &=k(3a-b)(3b-a)+2kb(3a-b)\left(-\frac{3}{2}\right)+2ka(3a-b) \left(\frac{1}{2}\right)\\
        &=0.
    \end{align*}
    \vspace{.05in}
    Now we check the smoothness of the Seifert $S^1$-bundles generated by the $D_i$. We have assumed that $\bar{H}$ generates a smooth Seifert bundle over $X$. Now, $\Tilde{X}$ might still have other $A_n$ singularities, but observe that $\mathcal{D}_i$ restricts to $\pi^*\bar{H}$ over those singular points, and by assumption they generate a smooth Seifert $S^1$-bundle over those singular points. The rest of the points are smooth, so $\mathcal{D}_i$ trivially generate the local class group since these groups are trivial.\\
    \vspace{.05in}
    
    The last part of the first case, when $\pi_1:\Tilde{X}\rightarrow \bar{X}$ is the resolution of an $A_3$ singularity, is to check if the Seifert $S^1$-bundles generated by $\mathcal{D}_i$ and their pullbacks are simply connected. 
Using Proposition 5.2 and Remark 5.3 from \cite{K05} we see that we can
check the vanishing of first homology and, by Proposition 5.1 in
\cite{K05}, also the simply-connectedness of the first Seifert
$S^1$-bundle over $\widetilde X$ by verifying that $\mathcal D_0$ is
primitive, namely that there exist
$\alpha,\beta\in H^2(\widetilde X,\mathbb Z)$ such that
$\gcd(\mathcal D_0\cdot\alpha,\mathcal D_0\cdot\beta)=1,$
or simply $\mathcal D_0\cdot\alpha=1$.
For the successive $S^1$-bundles we use the corresponding primitive
classes described below.
Note that the same argument is used in the proof of
\cite[Corollary~2.7]{FGV}.
Also we can consider a $T^3$-bundle as a sequence of three
$S^1$-bundles.
 We start with $\mathcal{D}_0$, consider $-C_2$
    \begin{align*}
        \mathcal{D}_0\cdot (-C_2)&=-2C_1\cdot (-C_2)=2C_1\cdot C_2=1;
    \end{align*}
    thus $\mathcal{D}_0=\pi^*\bar{H}-(k(3b-a)-1)(3a-b)C_0-2C_1$ defines a simply connected Seifert $S^1$-bundle. \\
    For $\mathcal{D}_1=\pi^*\bar{H}-2k(3b-a)C_1$ consider $\pi^*\bar{H}$, and $-C_2$
    \begin{align*}
        \mathcal{D}_1\cdot \pi^*\bar{H}&=c,\\
        \mathcal{D}_1\cdot (-C_2)&=2k(3b-a)C_1\cdot C_2=k(3b-a).\\
    \end{align*}
    Observe that $a=3m+c-1$ and $b=c+m$ for some $m\in \mathbb{N}$ we have 
    $$3b-a=3c+3m-3m-c+1=2c+1,$$
    so $3b-a$ is coprime to $c$ and we can make $k$ coprime to $c$ as well, thus $\text{gcd}(c,k(3b-a))=1.$
    Finally, for $\mathcal{D}_2=\pi^*\bar{H}-2k(3a-b)C_2$ consider $\pi^*\bar{H}$, and $-C_1$
    \begin{align*}
        \mathcal{D}_2\cdot \pi^*\bar{H}&=c,\\
        \mathcal{D}_2\cdot (-C_1)&=2k(3a-b)C_2\cdot C_1=k(3a-b).\\
    \end{align*}
    Observe that 
    $$3a-b=9m+3c-3-c-m=8m+2c-3,$$
    take 
    $$3a-b \mod c = 8m-3\mod c,$$
    by Dirichlet's theorem on arithmetic progressions there is an $m\in \mathbb{N}$ such that $8m-3$ is prime, making $3a-b$ coprime to $c,$ \\
    i.e. $\text{gcd}(c,k(3a-b))=1.$
    
For the successive circle bundles, set
\[
\mathcal P_1:=(3a-b)C_0-2C_1,\qquad
\mathcal P_2:=(a+b)C_0-(C_1+C_2).
\]
These are Cartier classes and $
\mathcal D_1-\mathcal D_0
=\bigl(k(3b-a)-1\bigr)\mathcal P_1,
$
while
\[
\mathcal D_2
=\mathcal D_0-\bigl(1+k(3a-b)\bigr)\mathcal P_1
+2k(3a-b)\mathcal P_2.
\]
Hence
\[
f_0^*\mathcal D_1
=\bigl(k(3b-a)-1\bigr)f_0^*\mathcal P_1,
\]
and, taking $f_0^*\mathcal P_1$ as the Chern class of the second
circle bundle,
\[
(f_0\circ f_1)^*\mathcal D_2
=2k(3a-b)(f_0\circ f_1)^*\mathcal P_2.
\]
Since $\mathcal G=C_0$, $\mathcal L=2C_1$ and
$\mathcal L'=C_1+C_2$ are integral generators, the definitions of
$\mathcal P_1$ and $\mathcal P_2$ show that
$f_0^*\mathcal P_1$ and
$(f_0\circ f_1)^*\mathcal P_2$ are primitive.
Moreover, the identities above and
$\mathcal E\cdot\mathcal D_i=0$ imply
$\mathcal E\cdot\mathcal P_i=0$.
Replacing $\mathcal D_1,\mathcal D_2$ by
$\mathcal P_1,\mathcal P_2$, respectively, and keeping the notation
$\mathcal D_1,\mathcal D_2$, gives the classes required in the lemma.

    \vspace{.05in}

  Now we check the case when $\pi_1:\Tilde{X}\rightarrow \bar{X}$ is the resolution of an $A_4$ singularity. Thus $\text{Pic}(\Tilde{X})\cong \text{Pic}(\bar{X})\oplus \langle \mathcal{L}_0,\mathcal{L}_0',\mathcal{L}_1,\mathcal{L}_1'\rangle$, where  $\mathcal L_0=C_0,
\mathcal L_0'=C_1,
\mathcal L_1=3C_2,$ and 
$\mathcal L_1'=C_2+C_3.$
  $C_3$ and $C_2$ are two irreducible curves in the exceptional divisor of the first blow-up in the resolution of the $A_4$ singularity, and $C_1$ and $C_0$ are two irreducible curves in the exceptional divisor of the second blow-up of the resolution,  i.e. of an $A_2$ singularity.  Observe that $C_0^2=C_1^2=-2$, $C_0\cdot C_1=1$, $C_2^2=C_3^2=-\frac{4}{3}$, $C_2\cdot C_3=\frac{1}{3}$, $C_0\cdot C_2=C_0\cdot C_3=C_1\cdot C_2=C_1\cdot C_3=0$. Let $c=\bar{H}^2$, and $\bar{\mathcal{E}}$ be an ample divisor over $\bar{X}$ defined as above, so $\bar{\mathcal{E}}\cdot \pi_0^*\bar{H} = \bar{H}^2=c$. We define 
    \begin{align*}
    \mathcal{E}&:=\frac{k(2x_0-x_1)(2x_1-x_0)(4x_2-x_3)}{c}\pi_1^*\bar{\mathcal{E}}-x_3C_3-x_2C_2-x_1C_1-x_0C_0,\\
    \mathcal{D}_0&:=\pi^*\bar{H}-k(2x_1-x_0)(4x_2-x_3)C_0,\\
    \mathcal{D}_1&:=\pi^*\bar{H}-k(2x_0-x_1)(4x_2-x_3)C_1,\\
    \mathcal{D}_2&:=\pi^*\bar{H}-3k(2x_0-x_1)(2x_1-x_0)C_2,
    \end{align*}
    where $x_3=4m+c-1$, $x_1=2m+c-1$, and $x_2=x_0=c+m$,  with $k,m\in \mathbb{N}$  to be determined later. Observe that $4x_3>x_2$, $4x_2>x_3$, $2x_1>x_0$, and $2x_0>x_1$. Also, note that $C_0, C_1$ are each Cartier, and $C_2$ is $\mathbb{Q}$-Cartier, with $3C_2$ Cartier. Thus $\mathcal{D}_0, \mathcal{D}_1, \mathcal{D}_2$ are Cartier since they are integer linear combinations of Cartier divisors.\\
    \vspace{.05in}\\
As in the $A_3$ case, we check the ampleness of $\mathcal E$ using the
Nakai--Moishezon criterion. Let $\widetilde D\subset\widetilde X$ be the
strict transform of an irreducible curve $D\subset\bar X$. If $D$ does
not pass through the $A_4$ singularity, then
$\widetilde D=\pi_1^*D$
and hence
\[
\mathcal E\cdot\widetilde D
=q\,\bar{\mathcal E}\cdot D>0,
\qquad
q=\frac{k(2x_0-x_1)(2x_1-x_0)(4x_2-x_3)}{c}.
\]

Suppose now that $D$ passes through the $A_4$ singularity. Its strict
transform can be written as
\[
\widetilde D=\pi_1^*D-\sum_{i=0}^3\mu_iC_i,
\]
where the $\mu_i$ are the multiplicities of the exceptional components
in the total transform of $D$. Therefore,
\[
\mathcal E\cdot\widetilde D
=q\,\bar{\mathcal E}\cdot D
-\sum_{i=0}^3\mu_i\,\mathcal E\cdot C_i.
\]
As in the $A_3$ case, there exists a constant $A>0$, independent of
$D$, such that
\[
\sum_{i=0}^3\mu_i\,\mathcal E\cdot C_i
\leq A\,\bar{\mathcal E}\cdot D.
\]
Consequently,
\[
\mathcal E\cdot\widetilde D
\geq(q-A)\bar{\mathcal E}\cdot D>0
\]
for $k$, and hence $q$, sufficiently large.

    Now we check the curves in the exceptional divisor and we get
    \begin{align*}
        \mathcal{E}\cdot C_0&=-x_1C_0\cdot C_1 - x_0C_0^2\\
        &=2x_0-x_1>0,\\
        \mathcal{E}\cdot C_1&=-x_1C_1^2-x_0C_0\cdot C_1\\
        &=2x_1-x_0>0,\\
        \mathcal{E}\cdot C_2&=-x_3C_3\cdot C_2-x_2C_2^2\\
        &=\frac{4x_2-x_3}{3}>0.\\
  \mathcal E\cdot C_3
&=-x_3C_3^2-x_2C_2\cdot C_3\\
&=\frac{4x_3-x_2}{3}>0.
\end{align*}
    Finally we check if $\mathcal{E}^2>0.$ Indeed,
   \begin{align*}
 \mathcal E^2
&=q^2\bar{\mathcal E}^{\,2}
+x_3^2C_3^2+x_2^2C_2^2+x_1^2C_1^2+x_0^2C_0^2
+2x_3x_2C_3\cdot C_2+2x_1x_0C_1\cdot C_0\\
&=q^2\bar{\mathcal E}^{\,2}
-\frac43(x_3^2+x_2^2)-2(x_1^2+x_0^2)
+\frac23x_3x_2+2x_1x_0>0,
    \end{align*}
 where the last inequality holds for $k$, and hence $q$, sufficiently large.
    Thus we have shown that $\mathcal{E}$ is ample.
        Now we show that for all $i$, $\mathcal{D}_i$ is primitive with respect to  $\mathcal{E}$. To verify this, we compute the intersection numbers explicitly:
    \begin{align*}
        \mathcal{E}\cdot \mathcal{D}_0&=q\bar{H}^2+x_0k(2x_1-x_0)(4x_2-x_3)C_0^2+x_1k(2x_1-x_0)(4x_2-x_3)C_1\cdot C_0\\
        &=qc-2x_0k(2x_1-x_0)(4x_2-x_3)+x_1k(2x_1-x_0)(4x_2-x_3)\\
        &=k(2x_0-x_1)(2x_1-x_0)(4x_2-x_3)+k(2x_1-x_0)(4x_2-x_3)(x_1-2x_0)\\
        &=0,\\
        \mathcal{E}\cdot \mathcal{D}_1&=q\bar{H}^2+ x_1 k(2x_0-x_1)(4x_2-x_3)C_1^2+x_0k(2x_0-x_1)(4x_2-x_3)C_1\cdot C_0\\
        &=qc-2x_1k(2x_0-x_1)(4x_2-x_3)+x_0k(2x_0-x_1)(4x_2-x_3)\\
        &=k(2x_1-x_0)(2x_0-x_1)(4x_2-x_3)+k(2x_0-x_1)(4x_2-x_3)(x_0-2x_1)\\
        &=0,\\
        \mathcal{E}\cdot \mathcal{D}_2&=qc+x_33k(2x_0-x_1)(2x_1-x_0)C_3\cdot C_2+x_23k(2x_0-x_1)(2x_1-x_0)C_2^2\\
        &=k(2x_1-x_0)(2x_0-x_1)(4x_2-x_3)+k(x_3-4x_2)(2x_0-x_1)(2x_1-x_0)\\
        &=0.
    \end{align*}
    This shows that $\mathcal{D}_i$ is primitive with respect to $\mathcal{E}$ for all $i$.\\
    \vspace{.05in}\\
    The smoothness of the Seifert $S^1$-bundles follows by the same argument as in the $A_3$ case.\\
 Finally, as in the $A_3$ case, we first check that the classes
$\mathcal D_i$ are primitive on $\widetilde X$.
Observe that for all $i$ we have $\mathcal{D}_i\cdot \pi^*\bar{H}=\bar{H}^2=c$, thus we have to find for each $\mathcal{D}_i$ a divisor whose intersection with $\mathcal{D}_i$ is coprime to $c$.\\
    Recall that $x_3=4m+c-1$, $x_1=2m+c-1$, and $x_2=x_0=c+m$.

Consider for $\mathcal D_0$ the divisor $-C_1$. We have
\[
\mathcal D_0\cdot(-C_1)
 =k(2x_1-x_0)(4x_2-x_3)
 =k(3m+c-2)(3c+1).
\]
Modulo $c$, this is congruent to $k(3m-2)$. By Dirichlet's theorem
on arithmetic progressions, we may choose $m\in\mathbb N$ such that
$3m-2$ is prime and does not divide $c$. We then choose $k$ such that
\[
k(3m-2)\equiv 1\pmod c.
\]
In particular, $\mathcal D_0\cdot(-C_1)$ is coprime to $c$.
Consider for $\mathcal D_1$ the divisor $-C_0$. Then
\[
\mathcal D_1\cdot(-C_0)
 =k(2x_0-x_1)(4x_2-x_3)
 =k(c+1)(3c+1)
 \equiv k \pmod c.
\]
Since $k(3m-2)\equiv 1\pmod c$, we have $\gcd(k,c)=1$, and hence
$\mathcal D_1\cdot(-C_0)$ is coprime to $c$.
Finally, consider for $\mathcal D_2$ the divisor $-C_3$. Since
$C_2\cdot C_3=\frac13$, we obtain
\[
\begin{aligned}
\mathcal D_2\cdot(-C_3)
 &=3k(2x_0-x_1)(2x_1-x_0)\,C_2\cdot C_3\\
 &=k(2x_0-x_1)(2x_1-x_0)\\
 &=k(c+1)(3m+c-2)\\
 &\equiv k(3m-2)\equiv 1\pmod c.
\end{aligned}
\]
Thus, $\mathcal D_2\cdot(-C_3)$ is coprime to $c$.
For the successive circle bundles, put
\[
A:=2x_0-x_1,\quad B:=2x_1-x_0,\quad C:=4x_2-x_3,
\]
and let $g:=\gcd(A,B),\qquad
A_0:=\frac{A}{g},\qquad B_0:=\frac{B}{g},\qquad
d:=\gcd(g,C).$
Choose $r,s\in\mathbb Z$ such that
$A_0r+B_0s=1,$
and set
\[
\mathcal P_1:=B_0C_0-A_0C_1,
\qquad
\mathcal P_2:=
-\frac{C}{d}rC_0-\frac{C}{d}sC_1+\frac{g}{d}(3C_2).
\]
These are Cartier classes and
$\mathcal D_1-\mathcal D_0=kCg\,\mathcal P_1,$
while
\[
\mathcal D_2
=\mathcal D_0+kBCs\,\mathcal P_1
-\frac{kABd}{g}\mathcal P_2.
\]
Hence $f_0^*\mathcal D_1=kCg\,f_0^*\mathcal P_1,$
and, taking $f_0^*\mathcal P_1$ as the Chern class of the second
circle bundle,
\[
(f_0\circ f_1)^*\mathcal D_2
=-\frac{kABd}{g}(f_0\circ f_1)^*\mathcal P_2.
\]
Since $\gcd(A_0,B_0)=1$, $f_0^*\mathcal P_1$ is primitive.
Moreover, $\mathcal P_1$ and $rC_0+sC_1$ form an integral basis
of the lattice generated by $C_0,C_1$, and
\[
\mathcal P_2
=-\frac{C}{d}(rC_0+sC_1)+\frac{g}{d}(3C_2).
\]
Since $\gcd(C/d,g/d)=1$, the class of $\mathcal P_2$ is primitive
after passing to the second circle bundle.
Thus the successive circle bundles are simply connected.
The identities above and $\mathcal E\cdot\mathcal D_i=0$ also imply
$\mathcal E\cdot\mathcal P_i=0$.
Replacing $\mathcal D_1,\mathcal D_2$ by
$\mathcal P_1,\mathcal P_2$, respectively, and keeping the notation
$\mathcal D_1,\mathcal D_2$, proves the statement.
   \end{proof}

For the solution of the  $G_2$ Hull-Strominger system we also need a vector bundle coming from  a stable bundle on $\tilde{X}_k$. The following Lemma provides such bundles, and its proof appears in Lemma 4.1 in \cite{FGM}.

\begin{lm}\label{lm4.2}
    Let $\tilde{X}_k$ be the orbisurface obtained by blowing up singular points of $X_{36},X_{50}$, such that $0\leq k\leq 18$ with $b_2(\tilde{X}_k)=k+4$. Then for any $k>1$ there exists on $\tilde{X}_k$  a stable bundle $V$ of rank $2$ with $c_1(V)=0$ and $c_2(V)=c$ for any $c\geq 5$.
\end{lm}

\section{Main results}

The main existence theorem of this note is the following.
\begin{te}\label{main}
Let $X$ be a K3 orbisurface satisfying the conditions in Lemma \ref{lm4.2} and Lemma \ref{lm4.1}. Let  $\beta_{1},\beta_{2},\beta_{3}$ be closed anti\mbox{-}self\mbox{-}dual $2$-forms such that $\Big[\tfrac{1}{2\pi}\beta_{j}\Big]\in H^{2}_{orb}(X,\mathbb{Z})\quad (j=1,2,3),$ which exist  by Lemma \ref{lm4.1}, and let $\pi:M\to X$ be the associated $T^{3}$–bundle. 
If $V$ is a stable vector bundle of rank $2$ on $X$
with $c_{1}(V)=0$ and $c_2(V)=c\geq 5$,
then there exists  $t>0$ and a smooth function $u$ on $X$ and a product hyperholomorphic connection $\theta_{X} = \nabla\times A$ on \ $T X^{1,0}\oplus V$ such that
$(\phi_{u,t},\pi^*\nabla,\pi^*A)$
solves the $G_{2}$ Hull-Strominger system  \eqref{StromingersystemG2} with $$\phi=\phi_{u,t}
= t^{3}\, \theta_{1}\wedge\theta_{2}\wedge\theta_{3}
- t\,e^{u}\big(\theta_{1}\wedge\omega_{1}
              + \theta_{2} \wedge\omega_{2}
              + \theta_{3} \wedge\omega_{3}\big).$$
\end{te}

\vspace{.1in}
\begin{proof}

Note that from  Lemma \ref{lm4.2} we know that stable bundles $V$ with $c_1(V)=0$ and $c_2(V)=c$ for any $c\geq 5$ exist over $X$, since $X$ is  $X_{36}$ or $X_{50}$ with  an  appropriate partial resolution of their singularities.   
Following the arguments in \cite{CGT}, the bundle $TX^{1,0} \oplus  V$  is polystable and $c_1(TX^{1,0}\oplus V)=c_1(TX^{1,0})+c_1(V)=0$. 
Since the orbifold singularities of $X$ are isolated Du~Val singularities,
the anti-self-duality and Hermitian--Yang--Mills conditions are local on the
orbifold charts. By the discussion at the end of Section~3, Verbitsky's
construction  can be applied in this orbifold setting.  Hence
$TX^{1,0}\oplus V$ admits an anti-self-dual connection $\theta_X$, which is
Hermitian--Yang--Mills with respect to each complex structure on $X$
\cite{Verbitsky}.
In particular we are interested in  the  product of Chern connections $\theta_X=\nabla\times A$ and curvatures $F_{\theta_X}=(R_\nabla,F_A)$.
Consider the bundle $E:=\pi^*(TX^{1,0}\oplus V)$ over the $T^3$ bundle $M$, and the connection $\theta:=\pi^*(\theta_X)$ on $E$. The curvature $F_\theta=\pi^*F_{\theta_X}$ is a $\phi_{u,t}$-instanton for any $u\in C^\infty_{orb}(X)$ and $t>0$. Thus, the Bianchi identity becomes an equation over $X$ 
$$\Delta h=t^2(|\beta_1|^2+|\beta_2|^2+|\beta_3|^2)+*_{4}\langle F_{\theta_X}\wedge F_{\theta_X}\rangle,$$
where we set $h=e^u.$ By general elliptic theory the equation above has a solution if and only if the integral of the right hand side is equal to zero. 

We equip $T X^{1,0}$ and $V$ with Hermitian metrics, so that their structure
groups reduce to compact Lie groups with Lie algebras $\mathfrak g_{\ell_1}$
and $\mathfrak g_{\ell_2}$, respectively. We fix invariant Hermitian inner
products $\operatorname{tr}_{\mathfrak g_{\ell_1}}$ and
$\operatorname{tr}_{\mathfrak g_{\ell_2}}$ on these Lie algebras extending the
Killing form. On $T X^{1,0}\oplus V$ we use the pairing
 \[
\langle\ ,\ \rangle
 = \frac{\alpha}{4}\bigl(\operatorname{tr}_{\mathfrak g_{\ell_2}}
   - \operatorname{tr}_{\mathfrak g_{\ell_1}}\bigr),
\]
for some $\alpha\in\mathbb R^*$. The induced quadratic curvature expression is
\[
\langle F_{\theta_X}\wedge F_{\theta_X}\rangle
 = \frac{\alpha}{4}\bigl(\operatorname{tr} F_A\wedge F_A
 - \operatorname{tr} R_\nabla\wedge R_\nabla\bigr).
\]

On cohomology, since $c_1(TX^{1,0}) = c_1(V)=0$, Chern–Weil theory gives
\[
\big[\langle F_{\theta_X}\wedge F_{\theta_X}\rangle\big]
=
2\pi^2\alpha
 \big(c_2(TX^{1,0})-c_2(V)\big)
\in H^4_{\mathrm{orb}}(X,\mathbb R).
\]

In particular, identifying $c_2(TX^{1,0})$ with the orbifold second Chern
class $c_2^{\mathrm{orb}}(X)$, we obtain
 \[
\frac{1}{2\pi^2}\int_X \langle F_{\theta_X}\wedge F_{\theta_X}\rangle
 = \alpha( c_2(TX^{1,0})- c_2(V) )=\alpha(e_{orb}(X)-c_2(V) ).
\]
Note that this number will not be an integer multiple of $\alpha$, because $$e_{orb}(X)=e(X)-\sum_{i}\frac{n_i}{n_i+1},$$ where the $n_i$'s are the indices of the $A_n$ singularities of $X$. 

We denote the intersection form on the second cohomology 
$$Q:H^2(X,\mathbb{Z})\times H^2(X,\mathbb{Z})\rightarrow\mathbb{Z};$$ 
we do not have to go to the orbifold cohomology since we have that $[\frac{1}{2\pi}\beta_j]$ are first Chern classes of the divisors $\mathcal{D}_j$ from  Lemma \ref{lm4.1} i.e.  $c_1(\mathcal{D}_j)=[\frac{1}{2\pi}\beta_j].$ So, we have 
$$Q\left(\left[\frac{1}{2\pi}\beta_j\right]\right):=Q(c_1(\mathcal{D}_j),c_1(\mathcal{D}_j))=\frac{-1}{4\pi^2}\int_X|\beta_j|^2\text{dvol}_X.$$
So, as we said before, to solve the equation $$\Delta h=t^2(|\beta_1|^2+|\beta_2|^2+|\beta_3|^2)+*_{4}\langle F_{\theta_X}\wedge F_{\theta_X}\rangle$$ we need the following numbers to be equal

$$\frac{\alpha}{2}(e_{orb}(X)-c_2(V))=t^2\sum_{j}Q(c_1(\mathcal{D}_j)).$$

We obtain  a solution for every

$$t=\sqrt{\frac{\frac{\alpha}{2}(e_{orb}(X)-c_2(V))}{\sum_{j}Q(c_1(\mathcal{D}_j))}}=\sqrt{\frac{\frac{\alpha}{2}(e_{orb}(X)-c)}{\sum_{j}Q(c_1(\mathcal{D}_j))}},$$
for any $c\geq 5$ by Lemma  \ref{lm4.2}.

\end{proof}

As a consequence we obtain 
\begin{co}
 The total space of a $T^3$ principal bundle $M$ over a singular K3 as in Theorem \ref{main} admits a solution of the $G_2$ Hull-Strominger system \eqref{StromingersystemG2} with $G = SU(2)$.
\end{co}

For the proof we just notice that the structure group of the bundle $V$ in Theorem \ref{main} is $SU(2)$.

\section{Examples}

We consider here some examples based on Iano-Fletcher's list \cite{IF}.
The families relevant to the main existence result (Theorem~5.1) are
$X_{36}$ and $X_{50}$. We also describe the singularities of $X_{30}$ for
comparison with the analogous construction in \cite{FGM}. These families
were selected because they provide a large variety of spaces with
different topology carrying solutions to the Hull--Strominger system.

Throughout, when we write
$X_{30}\subset \mathbb{P}(5,6,8,11),
X_{36}\subset \mathbb{P}(7,8,9,12),
X_{50}\subset \mathbb{P}(7,8,10,25),$
we mean a generic (quasi-smooth, well-formed) hypersurface in the corresponding family: while special members can occur with $\rho(X)=\mathrm{rk}\,\mathrm{Pic}(X)>1$, we work away from these Noether–Lefschetz-type loci and assume the hypersurface is chosen generically so that no additional divisor classes appear beyond the expected ones. This genericity assumption is consistent with Belcastro’s thesis, where it is shown that for a generic surface in one of Reid’s 95 weighted K3 hypersurface families one has $\rho(S)=1$ \cite{Be}. Here, the subscript $d$ in $X_d$ denotes the degree of the hypersurface, and $\mathbb{P}(a_0,a_1,a_2,a_3)$ denotes the weighted projective space with weights $a_0,a_1,a_2,a_3$ on the homogeneous coordinates.

The singularities of $X_{36}$ are described in \cite{IF}.  Here we describe the singularities of  $X_{30}$ and $X_{50}$. When restricting the defining polynomial to a singular edge, we set the
remaining coordinates equal to zero and divide by the common nonzero monomial factor on the corresponding affine chart. This yields the
reduced equations appearing below.
 In general, a  well-formed $\mathbb{P}(a_0,a_1,a_2,a_3)$ has singularities on each of its affine pieces $x_i\neq 0$, which are isomorphic to $\mathbb{C}^2/\mathbb{Z}_{n+1}$ where $n+1$ is the weight of $x_i$. They have singular edges when the $gcd(a_i,a_j)=d>1$, and the singular edges are isomorphic to $\mathbb{P}(a_i,a_j)\cong \mathbb{P}(\frac{a_i}{d},\frac{a_j}{d})$.

The hypersurface $X_{30}\subset \mathbb{P}(5,6,8,11)$ is given by a general polynomial
$$f=w^6 + x^5 + yz^2 + xy^3 + w^2x^2y+wxyz$$
with $(w,x,y,z)$ having weights $(5,6,8,11),$
where for simplicity we consider at the moment all the coefficients in the monomials to be ones. This surface is well-formed and quasismooth (see \cite{IF}), so its only possible singularities are inherited from the ones in $\mathbb{P}(5,6,8,11)$. Since $f$ has the monomials $w^6$ and $x^5$, when we restrict to $w\neq 0$ and $x\neq 0$, $f$ restricts to a polynomial of the form $f(1,x,y,z)=1+...$ and $f(w,1,y,z)=1+...$ thus the singular points $(1:0:0:0),$ and $(0:1:0:0)$ are not in the surface. On the contrary, the polynomial $f$ does not have monomials only on $y$ and $z$, thus the singular points $(0:0:1:0),$ and $(0:0:0:1)$ are in the surface, which are of type $A_7$, and $A_{10}$ respectively. Along the singular edge $\mathbb{P}(6,8)\cong \mathbb{P}(3,4)$ we have  $f=x^4+y^3 $, and with Lemma 9.4 in \cite{IF} we calculate that $X_{30}$ intersects at $\lfloor \frac{12}{3\cdot4}\rfloor=1$ points. Following 5.15 in \cite {IF} this is an $A_1$ singularity since $1=gcd(6,8)-1$   from $\mathbb{P}(6,8)$. In order to ensure that rk Pic$(X_{30}) = 1$ we need to consider generic coefficients in front of all monomials, but this does not change the type of the singularities.

Similarly, for $X_{50}\subset\mathbb{P}(7,8,10,25),$
with homogeneous coordinates $(w,x,y,z)$ of weights $(7,8,10,25)$,
we can consider for simplicity the polynomial
\[
f=z^2+y^5+w^6x+x^5y+w^2x^2y^2+wxyz.
\]
By the same reasoning as before, we have the monomials $z^2$ and $y^5$, so $X_{50}$ does not contain the singular points $(0:0:1:0)$ and $(0:0:0:1)$, it does contain the points $(1:0:0:0)$ and $(0:1:0:0)$ with $A_6$ and $A_7$ singularities. Along the edge $\mathbb{P}(8,10)\cong \mathbb{P}(4,5)$, we have $f=y^4+x^5$, so $X_{50}$ intersects the singular edge at $\lfloor \frac{20}{4\cdot5}\rfloor=1$ point. Again from 5.15 in \cite{IF} we conclude that this singular point is of type $gcd(8,10)-1$, hence an $A_1$ singularity. Along the edge $\mathbb{P}(10,25)\cong \mathbb{P}(2,5)$, we have $f=z^2+y^5$ which intersects the singular edge at $\lfloor \frac{10}{2\cdot5}\rfloor=1$ point giving an $A_4$ singularity.

In summary, $X_{30}$ has $A_1,A_7,$ and $A_{10}$ singularities; $X_{36}$ has $A_6,A_7,A_3$, and $A_2$ singularities; $X_{50}$ has $A_6,A_7,A_1,$ and $A_4$ singularities. Also, considering generic coefficients all of them have the rank of the Picard group equal to one.

According to Theorem \ref{main} the corresponding blow-ups of enough singularities will lead to solutions of the $G_2$ Hull-Strominger system. In the table at the end we list the consecutive blow-ups for $X_{36}$ and $X_{50}$ with their Euler characteristics.

  Finally, we mention the rational homology of such spaces. Suppose $\tilde{X}_k$ is one of the consecutive blow-ups of $X_{36}$, $X_{30}$, or $X_{50}$ as in Section 4. Then we have:
\begin{lm}
Consider the total space of the bundle $f_2:Y_3\rightarrow Y_2$ in Lemma 4.1 where $Y_2$ is a $T^2$ Seifert bundle over $\tilde{X}_k$. Then $Y_2$ is diffeomorphic to one of  $\sharp_{r}(S^2\times S^4)\sharp_{r+1} (S^3\times S^3)$ for $5\leq r \leq 22$. Moreover, every such $r$ appears for an appropriate sequence of blow-ups of $X$ and the Betti numbers of $Y_3$ are
$b_1=0,b_2 = r -1, b_3=3r+2, b_4 = 3r +2,  b_5=r-1.$
\end{lm}

The proof follows from the exact sequence of principal $S^1$-bundles as in \cite{FGM}, noting that both $Y_2$ and $Y_3$ are smooth. As  a Corollary,  we obtain:

\begin{co} For $5 \le r \le 22$ there is a  smooth simply-connected compact $7$-manifold $M$  with
$$
b_2(M)=b_5(M)=r-1,
b_3(M)=b_4(M)=3r+2,
$$
carrying a solution  to  the $G_2$ Hull-Strominger system.
\end{co}

\begin{table}[h!]
\centering
\begin{tabular}{|c|c|c|c|}
\hline
\textbf{$e(X)$} & \textbf{Surface} & \textbf{Fully Resolved Singularities} & \textbf{Remaining Singularities} \\
\hline
9  & $X_{36}$ & $A_3$ & $A_2,A_6,A_7$\\
10 & $X_{50}$ & $A_4$ & $A_1,A_6,A_7$\\
11 & $X_{50}$ & $A_1,A_4$ & $A_6,A_7$\\
12 & $X_{50}$ & $A_6$ & $A_1,A_4,A_7$\\
13 & $X_{50}$ & $A_7$ & $A_1,A_4,A_6$\\
14 & $X_{50}$ & $A_1,A_7$ & $A_4,A_6$\\
15 & $X_{36}$ & $A_3,A_6$ & $A_2,A_7$\\
16 & $X_{50}$ & $A_4,A_6$ & $A_1,A_7$\\
17 & $X_{50}$ & $A_1,A_4,A_6$ & $A_7$\\
18 & $X_{50}$ & $A_1,A_4,A_7$ & $A_6$\\
19 & $X_{50}$ & $A_6,A_7$ & $A_1,A_4$\\
20 & $X_{50}$ & $A_1,A_6,A_7$ & $A_4$\\
21 & $X_{36}$ & $A_2,A_6,A_7$ & $A_3$\\
22 & $X_{36}$ & $A_3,A_6,A_7$ & $A_2$\\
23 & $X_{50}$ & $A_4,A_6,A_7$ & $A_1$\\
\hline
\end{tabular}
\caption{Table of $e(X)$ values with corresponding surfaces and fully resolved singularities.}
\end{table}

\newpage

\smallskip
\textbf{Acknowledgements}.
Anna Fino is partially supported by Project PRIN 2022 \lq \lq Geometry and Holomorphic Dynamics”, by GNSAGA (Indam) and by a grant from the Simons Foundation (\#944448).  Gueo Grantcharov is partially supported by a grant from the Simons Foundation (\#853269).

The authors thank the anonymous reviewer for his comments and questions
which helped to improve the presentation.
\medskip

\noindent {\bf Conflict of interest statement}

The authors declare that they have no conflicts of interest or competing interests.

\smallskip

\noindent {\bf Data availability statement}

No new data were created or analysed during this study. Data sharing is not applicable to this article.

\end{document}